\documentclass[11pt]{amsart}
\usepackage{geometry}
\usepackage{graphicx}
\usepackage{amssymb}
\usepackage{amsmath, mathrsfs, stmaryrd}
\usepackage{color,ulem}
\newtheorem{theorem}{Theorem}
\newtheorem{definition}[theorem]{Definition}

\newtheorem{lem}[theorem]{Lemma}

\newtheorem{remark}{Remark}

\newcommand{\R}{\mathbb R}
\newcommand{\pl}{\partial}
\numberwithin{equation}{section}

\newcommand{\na}{\nabla}
\newcommand{\vphi}{\varphi}
\newcommand{\lt}{\left}
\newcommand{\rt}{\right}
 \address{Department of Mathematics\\ Columbia University \\ New York, NY 10027, USA}
  \email{pnchen@math.columbia.edu}
 \address{Department of Mathematics\\ Columbia University \\ New York, NY 10027, USA}
  \email{mtwang@math.columbia.edu}
 \address{Department of Mathematics\\ Columbia University \\ New York, NY 10027, USA}
\curraddr{Department of Mathematics\\ Michigan State University \\East Lansing, MI 48824, USA}
  \email{yw2293@math.columbia.edu}
\begin{document}
\title[timeflat surface]{Rigidity of time-flat surfaces in the Minkowski spacetime}
\author{Po-Ning Chen, Mu-Tao Wang, and Ye-Kai Wang}
\begin{abstract}
A time-flat condition on spacelike 2-surfaces in spacetime is considered here. This condition is  analogous to the constant torsion condition for curves in a three-dimensional space and has been studied in \cite{Bray-Jauregui, Chen-Wang-Yau1, Chen-Wang-Yau2,  Wang-Yau1, Wang-Yau2}. In particular, any 2-surface in a static slice of a static spacetime is time-flat. In this article, we address the question in the title and prove several local and global rigidity theorems for such surfaces in the Minkowski and Schwarzschild spacetimes. Higher dimensional generalizations are also considered.
\end{abstract}

\maketitle

\section{Introduction} 
The geometry of spacelike 2-surfaces in spacetime plays a crucial role in general relativity. Penrose's singularity theorem predicts future singularity formation from the existence of a trapped 2-surface. A black hole is quasi-locally described by a marginally outer trapped 2-surface. These conditions can be expressed in terms of the mean curvature vector field $H$ of the 2-surface. $H$ is the unique normal vector field determined by the variation of the area functional and is ultimately connected to the warping of spacetime in the vicinity of the 2-surface. It is thus not surprising that several definitions of quasi-local mass in general relativity are closely related to the mean curvature vector field. In particular, both the Hawking mass \cite{Hawking} and the Brown-York-Liu-Yau mass \cite{Brown-York, Liu-Yau} involve the norm of the mean curvature vector field $|H|$. In the new definition of quasi-local mass in \cite{Wang-Yau1, Wang-Yau2}, in addition to $|H|$, the direction of the mean curvature vector field is
  also utilized. When the mean curvature vector field is spacelike everywhere on $\Sigma$ (thus $|H|>0$), the direction of $H$  defines a connection one-form $\alpha_H$ of the normal bundle  (see Definition \ref{mean_curvature_gauge} for the precise definition of $\alpha_H$). The quasi-local mass in \cite{Wang-Yau1, Wang-Yau2} is defined in terms of the induced metric $\sigma$ on the surface, $|H|$, and $\alpha_H$. In particular, the condition \begin{equation}\label{timeflat} div_\sigma(\alpha_H)=0\end{equation}
 implies that the isometric  embedding of $\Sigma$ into $\R^3\subset \R^{3,1}$ is an optimal isometric embedding in the sense of \cite{Wang-Yau1, Wang-Yau2}. Recently, Bray and Jauregui \cite{Bray-Jauregui} discovered a very interesting monotonicity property of the Hawking mass along surfaces that satisfy the condition \eqref{timeflat}. Such surfaces are said to be ``time-flat" in \cite{Bray-Jauregui} and include all 2-surfaces in a time-symmetric initial data set.

 For curves in $\R^3$, the direction of the mean curvature vector corresponds to the normal of the curve \cite[pages 16-17]{do Carmo}. The connection 1-form $\alpha_H$ is nothing but the torsion of the curve (see Definition \ref{mean_curvature_gauge}) and condition (\ref{timeflat}) simply says the torsion is constant. 
Weiner \cite{Weiner1, Weiner2} constructed simple closed curves with constant torsion in $\R^3$ that do not lie in a totally geodesic $\R^2$. On the other hand, in \cite{Bray-Jauregui2}, Bray and Jauregui proved that a curve in $\R^3$ with constant torsion must lie in a totally geodesic $\R^2$ if it can be written as a graph over a simple closed curve in $\R^2$. 

 A natural rigidity question raised by Bray \cite{Bray} is ``Must a time-flat surface in the Minkowski spacetime lie in a totally geodesic $\R^3$?"
 From the above analogy between curves and surfaces, one expects some global condition is needed in order for the rigidity property of time-flat surfaces to hold.   
 In this article, we prove several global  and local rigidity theorems for time-flat surfaces in the Minkowski and Schwarzschild spacetimes under various conditions. 

We first prove a local rigidity theorem that holds in all dimensional Minkowski spacetimes:
\vskip 10pt
\noindent {\bf Theorem \ref{thm_local}.} {\it Let $n \ge 3$. Suppose $\Sigma$ is a mean convex hypersurface which lies in a totally geodesic $\R^n$ in the $(n+1)$-dimensional Minkowski spacetime $\R^{n,1}$, then $\Sigma$ is locally rigid as a time-flat $(n-1)$-dimensional submanifold in $\R^{n,1}$. In other words, any infinitesimal deformation of $\Sigma$ that preserves the time-flat condition must be a deformation in the $\R^n$ direction, a deformation that is induced by a Lorentz transformation of $\R^{n,1}$, or a combination of these two types of deformations.}
\vskip 10pt

We also prove three global rigidity theorems:
\vskip 10pt
\noindent {\bf Theorem \ref{global 2}.}
{\it Suppose $\Sigma$ is a time-flat 2-surface in $\R^{3,1}$ such that $\alpha_H =0$ and $\Sigma$ is a topological sphere, then $\Sigma$ lies in a totally geodesic hyperplane.}

\vskip10pt
\noindent {\bf Theorem \ref{global_axial}.}
{\it Suppose $\Sigma$ is a time-flat  2-surface in  $\R^{3,1}$ and $\Sigma$ is a topological sphere. If
$\Sigma$ is invariant under a rotational Killing vector field, then $\Sigma$ lies in a totally geodesic hyperplane in $\R^{3,1}$.}

\vskip 10pt

The last global rigidity theorem holds for the $(n+1)$-dimensional Schwarzschild spacetime with metric
\[ - (1- \frac{2m}{r^{n-2}}) dt^2 + \frac{1}{1-\frac{2m}{r^{n-2}}} dr^2 + r^2 g_{S^{n-1}}, \] where $m\geq 0$ is the mass and $g_{S^{n-1}}$ is the 
standard metric on a unit sphere $S^{n-1}$.  

\vskip 10pt
\noindent {\bf Theorem \ref{global3}.}
{\it Let $n \ge 3$ and $\Sigma^{n-1} $ be a connected spacelike codimension-2 submanifold in the $(n+1)$-dimensional Schwarzschild spacetime with mass $m$. Suppose $\alpha_H=0$ and $\Sigma$ is star-shaped with respect to $Q$ (see Definition \ref{star-shaped}) where $Q = rdr \wedge dt$.  Then, when $m>0$, $\Sigma$ lies in a time-slice; when $m=0$, $\Sigma$ lies in a totally geodesic hyperplane.}

\vskip 10pt
Note that $Q$ is a conformal Killing-Yano 2-form on the Schwarzschild spacetime, see \cite{JL}. Theorem \ref{thm_local} follows from applying the Reilly formula to the linearized equation \eqref{linearized}. Theorem \ref{global 2} is proved using the Codazzi equation and a theorem of Yau from \cite{Yau}.  The proof of Theorem \ref{global_axial} is reduced to Theorem \ref{global 2} using the structure of axially symmetric surfaces in $\R^{3,1}$. Theorem \ref{global3} is proved using a Minkowski type integral formula involving the conformal Killing-Yano 2-form on the Schwarzschild spacetime.
 
We review the geometry of spacetime surfaces in \S 2 and define the time-flat condition and the higher dimensional generalization. Theorem \ref{thm_local} is proved in \S 3, Theorem \ref{global 2} is proved in \S 4, Theorem \ref{global_axial} is proved in \S 5, and Theorem \ref{global3} is proved in \S 6.

\section{Geometry of spacelike 2-surface in spacetime}

Let $N$ be a time-oriented spacetime. Denote the Lorentzian metric on $N$ by $\langle \cdot, \cdot\rangle$ and covariant derivative by $\nabla^N$. Let $\Sigma$ be a closed space-like two-surface embedded in $N$. Denote the induced Riemannian metric on $\Sigma$ by $\sigma$ and the gradient and Laplace operator of $\sigma$ by $\nabla$ and $\Delta$, respectively.

 Given any two tangent vector $X$ and $Y$ of $\Sigma$, the second fundamental form of $\Sigma$ in $N$ is given by $\mbox{II}(X, Y)=(\nabla^N_{X} Y)^\perp$ where $(\cdot)^\perp$ denotes the projection onto the normal bundle of $\Sigma$. The mean curvature vector is the trace of the second fundamental form, or ${H}=tr_\Sigma \mbox{II}=\sum_{a=1}^2 \mbox{II}(e_a, e_a)$ where $\{e_1, e_2\}$ is an orthonormal basis of the tangent bundle of $\Sigma$.

 The normal bundle is of rank two
with structure group $SO(1,1)$ and the induced metric on the normal bundle is of signature $(-, +)$. Since the Lie algebra of $SO(1,1)$ is isomorphic to $\R$, the connection form of the normal bundle is a genuine 1-form that depends on the choice of the normal frames. The curvature of the normal bundle is then given by an exact 2-form which reflects the fact that any $SO(1,1)$ bundle is topologically trivial. Connections of different choices of normal frames differ by an exact form. We define (see \cite{Wang-Yau2}):
\begin{definition}
Let $e_3$ be a space-like unit normal along $\Sigma$, the connection one-form determined by $e_3$ is defined to be
\begin{equation}\label{connection}\alpha_{e_3} =\langle \nabla^N_{(\cdot)} e_3, e_4\rangle \end{equation} where $e_4$  is the future-directed time-like unit normal that is orthogonal to $e_3$.
\end{definition}

\begin{definition}\label{mean_curvature_gauge}
Suppose the mean curvature vector field $H$ of $\Sigma$ in $N$ is a spacelike vector field. The connection one-form in mean curvature gauge is \[\alpha_{H} =\langle \nabla^N_{(\cdot)} e_3, e_4\rangle,\] where $e_3=-\frac{H}{|H|}$ and $e_4$ is the future-directed timelike unit normal that is orthogonal to $e_3$.
\end{definition}

\begin{definition}
We say $\Sigma$ is time-flat if $div_\sigma( \alpha_H)=0.$
\end{definition}
\begin{remark}\label{higher dimension}
For spacelike codimension-2 submanifolds in a time-oriented $(n+1)$-dimensional spacetime $N$, the second fundamental form $\mbox{II}$ and the mean curvature vector $H$ can be defined in the same manner. Let $e_n = -\frac{H}{|H|}$ and $e_{n+1}$ be the future timelike normal orthogonal to $e_n$. The connection one-form with respect to mean curvature gauge is defined to be
\[ \alpha_H = \langle \na^N_{(\cdot)}e_n, e_{n+1} \rangle. \]
\end{remark}

\section{Local rigidity of mean convex hypersurfaces in $\R^n\subset \R^{n,1}$}
The local rigidity problem can be formulated as follows. Suppose $\Sigma$ is time-flat and is given by an embedding $X$. Suppose $V$ is a smooth vector field along $\Sigma$ such that the image of $X(s)=X+sV$ is infinitesimally time-flat in the sense the derivatives of  $div_\sigma( \alpha_H)$ along the image with respect to $s$ is zero when $s=0$. Do all such $V$ correspond to trivial deformations?

It is easy to see that submanifolds lying in a totally geodesic slice is time-flat. We assume $\partial\Omega = \Sigma \subset \{ t=0\} = \mathbb{R}^n$. It is clear that 
any deformation in the $\R^n$ direction preserves the time-flat condition. On the other hand, a Lorentz transformation preserves the geometry of $\Sigma$ and thus preserves the time-flat condition.

Let $\nabla, \Delta$ denote the covariant derivative and Laplacian of the induced metric $\sigma.$ Let $h_{ab},h$ be the second fundamental form and mean curvature of $\Sigma \subset \mathbb{R}^n$ with respect to the outward unit normal $\nu.$
\begin{theorem}\label{thm_local}
Let $n \ge 3$. Suppose $\Sigma$ is a mean convex hypersurface which lies in a totally geodesic $\R^n$ in the $(n+1)$-dimensional Minkowski spacetime $\R^{n,1}$, then $\Sigma$ is locally rigid as a time-flat $(n-1)$-dimensional submanifold in $\R^{n,1}$. In other words, any infinitesimal deformation of $\Sigma$ that preserves the time-flat condition must be a deformation in the $\R^n$ direction, a deformation that is induced by a Lorentz transformation of $\R^{n,1}$, or a combination of these two types of deformations.
\end{theorem}
\begin{proof}
In this proof, we denote $\alpha_H$  by $\alpha$. Since $\delta (div_\sigma \alpha)$ depends linearly on infinitesimal deformation and any deformation in $\R^n$ corresponds to trivial deformations, it suffices to consider deformations in the time direction. Let $V= f \frac{\partial}{\partial t}$ for a smooth function $f$ defined on $\Sigma$ be such an infinitesimal deformation and $X(s) = (\tau(s),X^1(s),\ldots,X^n(s))$ be the corresponding deformation. Since we only vary the surface in the time direction, $X^i(s) = X^i(0)$ for $i=1,\ldots,n,$ and $\delta\tau =f$. Here $\delta$ stands for $\frac{\partial}{\partial s}\big|_{s=0}$. 

We start by computing the variation of $div_\sigma \alpha$.  The induced metrics satisfy
\begin{align}
\sigma(s)_{ab} = \sigma_{ab} - \frac{\partial \tau(s)}{\partial u^a} \frac{\partial\tau(s)}{\partial u^b}.
\end{align}
Since $\tau(0)=0$, $\delta\sigma=0.$ Let $\Delta_s$ be the Laplacian of the induced metric on $X(s).$ We have
\begin{align}
H = (\Delta_s \tau(s), \Delta_s X^1, \ldots, \Delta_s X^n).
\end{align}
$\delta\sigma=0$ implies the infinitesimal variation of Laplacian is zero. Therefore, we have
\begin{align}
\delta H &= (\Delta f) \frac{\partial}{\partial t}\\
\delta |H|^2 &= 2 \langle \delta H, -h e_n \rangle = 0\\
\delta e_n &= -\frac{\delta H}{h} + \frac{\delta |H|}{h^2} H = - \frac{\Delta f}{h} \frac{\partial}{\partial t} \label{variation of e_n}
\end{align}
Since $0= \delta \langle e_{n+1}, \frac{\partial}{\partial u^a} \rangle = \delta\langle e_{n+1},e_n \rangle,$ we have
\[ 0=\langle \delta e_{n+1}, \frac{\pl}{\pl u^a} \rangle + \langle e_{n+1}, \frac{\pl f}{\pl u^a} \frac{\pl}{\pl t} \rangle \]
and
\[ 0=\langle \delta e_{n+1}, e_n \rangle + \langle e_{n+1}, -\frac{\Delta f}{h} \frac{\partial}{\partial t}\rangle. \]
Hence
\begin{align}\label{variation of e_n+1}
\delta e_{n+1} = \nabla f - \frac{\Delta f}{h} e_n
\end{align}
We are ready to compute the variation of $\alpha$.
\begin{align*}
(\delta \alpha)_a &= \delta \langle D_a e_n,e_{n+1} \rangle \\
&= \langle (\delta D)_a e_n,e_{n+1} \rangle + \langle D_{\frac{\pl (\delta X)}{\pl u^a} } e_n,e_{n+1} \rangle + \langle D_a (\delta e_n),e_{n+1} \rangle + \langle D_a e_n, \delta e_{n+1} \rangle.
\end{align*}
Since $\delta\sigma=0$, $\delta D=0$. By (\ref{variation of e_n}) and (\ref{variation of e_n+1}), we get
\begin{align*}
(\delta \alpha)_a &= \langle D_{\frac{\partial f}{\partial u^a} \frac{\partial}{\partial t}}e_n, e_{n+1}\rangle + \langle D_a \left( - \frac{\Delta f}{h} \right) \frac{\partial}{\partial t}, \frac{\partial}{\partial t} \rangle + \langle D_a e_n, \nabla f - \frac{\Delta f}{h} e_n \rangle\\
&= \nabla_a \left( \frac{\Delta f}{h}\right) + h_{ab} \nabla^b f 
\end{align*}
and
\begin{equation}\label{linearized}
\delta(div_\sigma \alpha) = \Delta \left( \frac{\Delta f}{h} \right) + \nabla^a \left( h_{ab} \nabla^b f\right).
\end{equation}
We remark that the linearization of this operator was also derived in \cite{Chen-Wang-Yau1, Miao-Tam-Xie}. 

To prove the theorem, it suffices to show that $f$ is the restriction of linear coordinate functions on $\Sigma$. 
Let $u$ solve the Dirichlet problem
\begin{align*}
\left\{ \begin{array}{cc}
\Delta u=0 & \mbox{in }\Omega\\
u=f & \mbox{on }\Sigma
\end{array}\right.
\end{align*}
On $\R^n$, the Reilly formula \cite{Reilly} reads 
\begin{align}\label{Reilly formula}
\int_\Omega |D^2 u|^2 = -\int_\Sigma \left( h^{ab} \nabla_a f \nabla_b f + 2 (\Delta f) e_n(u) + h (e_n(u))^2 \right).
\end{align}
This was used in \cite{Miao-Tam-Xie} to derive minimizing property of the Wang-Yau quasi-local energy. 
On the other hand, multiplying $\delta (div_\sigma \alpha) = 0$ by $f$ and integrating over $\Sigma$ yield
\begin{align}\label{integrated time-flat variation}
\int_\Sigma \left(\frac{(\Delta f)^2}{h} - h^{ab}\nabla_a f\nabla_b f \right)=0.
\end{align}
Adding \eqref{Reilly formula} and \eqref{integrated time-flat variation} together and completing square, we obtain
\begin{align*}
\int_\Omega |D^2 u|^2 + \int_\Sigma \left( \frac{\Delta f}{\sqrt{h}} + \sqrt{h} e_n(u) \right)^2 =0.
\end{align*}
Hence $u$ is a linear function up to a constant.
\end{proof}

\section{Global rigidity for surfaces with $\alpha_H=0$} 

We first recall a general sufficient condition of Yau that implies a submanifold lies in another submanifold which is totally geodesic. 
\begin{theorem}\cite[Theorem 1]{Yau}\label{Yau}
Let $M$ be a submanifold in a $p$-dimensional pseudo-Riemannian manifold  $P$ with constant curvature. Let $N_1$ be a subbundle of the normal bundle of $M$ with fiber dimension $k$. Suppose the second fundamental form of $M$ with respect to any direction in $N_1$ vanishes and $N_1$ is parallel in the normal bundle. Then $M$ lies in a $(p-k)$-dimensional totally geodesic submanifold.
\end{theorem}
Although in \cite{Yau}, the theorem is proved only for Riemannian manifold, the same argument works in the pseudo-Riemannian case. We apply Theorem 5 to prove the following global rigidity theorem.
\begin{theorem}\label{global 2}
Suppose $\Sigma$ is a time-flat 2-surface in $\R^{3,1}$ such that $\alpha_H =0$ and $\Sigma$ is a topological sphere, then $\Sigma$ lies in a totally geodesic $\R^3$.
\end{theorem}

\begin{proof}
Denote by $e_3=-\frac{H}{|H|}$ and $e_4$ to be the unit future timelike normal that is orthogonal to $e_3$. The second fundamental form of $\Sigma$ can be written as  $h^3_{ab} e_3 - h^4_{ab} e_4$ and $h^4_{ab}$ is trace-free. The Codazzi equation for $h^4_{ab}$ reads
\[ \nabla^a h^4_{ab}  - \nabla_b  tr_\sigma h^4 + (\alpha_H)^a h^3_{ab} - tr_\sigma h^3 (\alpha_H)_b=  0.  \]

Since  $\alpha_H =0$ and $h^4_{ab}$ is trace-free, this reduces to
\[ \nabla^a h^4_{ab} =  0. \]
A divergence-free symmetric trace-free 2-tensor corresponds to a holomorphic quadratic differential, which must vanish since $\Sigma$ is a topological sphere.  Let $N_1$ be the subbundle of the normal bundle spanned by $e_4$. Since $\alpha_H=0$, $N_1$ is parallel in the normal bundle. Hence by Theorem \ref{Yau} above, $\Sigma$ lies in a totally geodesic hyperplane in $\R^{3,1}$. 

\end{proof}

\section{Global rigidity for axially symmetric and time-flat surfaces in $\R^{3,1}$.}

In this section, we prove the following global rigidity theorem for time-flat axially symmetric surfaces in $\R^{3,1}$.

\begin{theorem}\label{global_axial}
Suppose $\Sigma$ is a time-flat  2-surface in  $\R^{3,1}$ and $\Sigma$ is a topological sphere. If
$\Sigma$ is invariant under a rotational Killing vector field, then $\Sigma$ lies in a totally geodesic hyperplane in $\R^{3,1}$.
\end{theorem}
\begin{proof}
Without loss of generality, we assume that the rotational Killing vector field is $\frac{\pl}{\pl\phi}$ where $(t,r,\theta,\phi)$ is the standard spherical coordinate in $\R^{n,1}$, and in terms of this coordinate system $\Sigma$ is locally given by the embedding
\[ F: (\theta,\phi) \rightarrow (t(\theta),r(\theta),\theta,\phi). \]
A basis of the tangent space of $\Sigma$ consists of 
\[ t' \frac{\pl}{\pl t} + r' \frac{\pl}{\pl r} + \frac{\pl}{\pl\theta} \mbox{ and } \frac{\pl}{\pl\phi} \]
where $()'$ stands for differentiation with respect to $\theta$. On the other hand, a basis of the normal bundle is given by
\[  Y_3 = r^2 \frac{\pl}{\pl r} - r' \frac{\pl}{\pl\theta} \mbox{ and } Y_4 = r' \frac{\pl}{\pl t} + t' \frac{\pl}{\pl r}. \]
By axial-symmetry, both normal vectors $e_3 = -\frac{H}{|H|}$ and its orthogonal complement $e_4$ can be written as linear combinations of $Y_3$ and $Y_4$ with coefficients depending only on $\theta$.
We compute
\[ \alpha_H (\frac{\pl}{\pl\phi})= \langle D_{\frac{\pl}{\pl\phi}} e_3,e_4 \rangle = \lt\langle a(\theta) \bar{\Gamma}_{\phi r}^\phi \frac{\pl}{\pl\phi} + b(\theta) \bar{\Gamma}_{\phi\theta}^\phi \frac{\pl}{\pl\phi}, e_4 \rt\rangle=0.\]
Hence $\alpha_H = \vphi(\theta) d\theta$ and $d\alpha_H=0$. Since $\Sigma$ is a topological 2-sphere, $d\alpha_H = div\alpha_H=0$ imply $\alpha_H=0$. By Theorem \ref{global 2}, $\Sigma$ lies in a totally geodesic hyperplane. 
\end{proof}

\section{Global Rigidity of Codimension-2 Submanifolds with $\alpha_H=0$ in the Schwarzschild Spacetime}
We generalize Theorem \ref{global 2} to ``star-shaped" (see Definition \ref{star-shaped}) codimension-2 submanifolds  in the $(n+1)$-dimensional Schwarzschild spacetime with mass $m \ge 0$. In Schwarzschild coordinates, the metric $\bar {g}$ is of the form
\[ \bar{g} = - (1- \frac{2m}{r^{n-2}}) dt^2 + \frac{1}{1-\frac{2m}{r^{n-2}}} dr^2 + r^2 g_{S^{n-1}}. \]
We include the case $m=0$ which corresponds to the Minkowski spacetime. 

Let $Q = r dr \wedge dt$ be the conformal Killing-Yano 2-form and $Q^2$ be the symmetric 2-tensor given by 
\[ (Q^2)_{\alpha \beta}= Q_{\alpha}^{\,\,\,\,\gamma}Q_{\gamma \beta}.\]
We need the following lemma from \cite[Lemma B.1]{thesis} relating the curvature tensor and $Q$. 
\begin{lem}
The curvature tensor $\bar{R}_{\alpha\beta\gamma\delta}$ of the Schwarzschild metric $\bar{g}$ can be expressed as
\begin{align}
\begin{split}
\bar{R}_{\alpha\beta\gamma\delta} &= \frac{2m}{r^n} \lt( \bar{g}_{\alpha\gamma} \bar{g}_{\beta\delta} - \bar{g}_{\alpha\delta} \bar{g}_{\beta\gamma} \rt) - \frac{n(n-1)m}{r^{n+2}} \lt( \frac{2}{3}	Q_{\alpha\beta} Q_{\gamma\delta} - \frac{1}{3} Q_{\alpha\gamma}Q_{\delta\beta}-\frac{1}{3}Q_{\alpha\delta}Q_{\beta\gamma} \rt) \\
&\quad - \frac{nm}{r^{n+2}} \Big( \bar{g}\circ Q^2 \Big)_{\alpha\beta\gamma\delta} \label{curvature tensor in terms of Q: Schwarzschild}
\end{split}
\end{align}
where $(\bar{g}\circ Q^2)_{\alpha\beta\gamma\delta} = \bar{g}_{\alpha\gamma} (Q^2)_{\beta\delta} - \bar{g}_{\alpha\delta} (Q^2)_{\beta\gamma} + \bar{g}_{\beta\delta} (Q^2)_{\alpha\gamma} - \bar{g}_{\beta\gamma} (Q^2)_{\alpha\delta}$.
\end{lem}

Let $\Sigma$ be a spacelike codimension-2 submanifold in the Schwarzschild spacetime of $(n+1)$-dimension. Let $e_n$,  $e_{n+1}$ and $\alpha_H$ be as in Remark \ref{higher dimension}. We consider a natural condition that generalizes the star-shaped condition for hypersurfaces in the Euclidean space. 
\begin{definition}\label{star-shaped}
$\Sigma$ is said to be star-shaped with respect to $Q$ if 
$Q(e_n,e_{n+1}) >0$.

\end{definition}

We are now ready to prove the main theorem in this section: 
\begin{theorem}\label{global3}
Let $n \ge 3$ and $\Sigma^{n-1} $ be a connected spacelike codimension-2 submanifold in the $(n+1)$-dimensional Schwarzschild spacetime with mass $m\geq 0$. Suppose $\alpha_H=0$ and $\Sigma$ is star-shaped with respect to $Q$ where $Q = rdr \wedge dt$.  Then, when $m>0$, $\Sigma$ lies in a time-slice; when $m=0$, $\Sigma$ lies in a totally geodesic hyperplane.
\end{theorem}
\begin{proof} Let $\sigma$ denote the induced metric on $\Sigma$. Let $h_n$ and $h_{n+1}$ be the second fundamental forms with respect to $e_n$ and $e_{n+1}$, respectively. We pick a tangential basis $\partial_a, a=1,\cdots n-1$ and express the components of a tensor in terms of this basis together with $e_n$ and $e_{n+1}$. For example, $Q_{nb}=Q(e_n, \partial_b)$. In the following computations, we lower and raise indices with respect to the induced metric $\sigma_{ab}=\sigma(\partial_a,\partial_b)$ and its inverse $\sigma^{ab}$.  Consider the divergence quantity on $\Sigma$:
\begin{equation}\label{divergence} \na_a \lt[ \lt( \mbox{tr}_\sigma(h_{n+1}) \sigma^{ab} - h_{n+1}^{\;\;\;\;\;\;ab} \rt) Q_{nb} \rt]. \end{equation}

By the assumption $\alpha_H=0$ and the Codazzi equation (see for example \cite[Theorem 2.1]{thesis}), we obtain 
\begin{align*}
\na_a \lt( \mbox{tr}_\sigma(h_{n+1}) \sigma^{ab} - h_{n+1}^{\;\;\;\;\;\;ab} \rt) = -\bar{R}^{ab}_{\;\;\;\;a,n+1}. 
\end{align*}
On the other hand, $Q$ satisfies the conformal Killing-Yano equation \cite[Definition 1]{JL}
\[ D_\alpha Q_{\beta\gamma} + D_\beta Q_{\alpha\gamma} = \frac{2}{n} \lt( \bar{g}_{\alpha\beta} \xi_\gamma - \frac{1}{2} \bar{g}_{\alpha\gamma} \xi_\beta - \frac{1}{2} \bar{g}_{\beta\gamma} \xi_\alpha \rt)\]
where $\xi^\beta = D_\alpha Q^{\alpha\beta}$. In our case $Q = rdr \wedge dt$ and $\xi = -n \frac{\pl}{\pl t}$. Therefore, 
\begin{align*}
   &\frac{1}{2} \Big( \na_a (Q(e_n,\pl_b)) + \na_b (Q(e_n,\pl_a)) \Big) \\
= & \sigma_{ab} \lt\langle \frac{\pl}{\pl t}, e_n \rt\rangle - h_{n+1,ab}Q_{n+1,n}  -\frac{1}{2} ( Q_{bc} h_{na}^{\;\;\;\;c} + Q_{ac} h_{nb}^c ). 
\end{align*}

The assumption $\alpha_H=0$ and Ricci equation (see for example \cite[Theorem 2.1]{thesis}) together imply
\begin{align*}
Q_{bc} h_{na}^{\;\;\;\;c} h_{n+1}^{\;\;\;\;\;\;ab} = \frac{1}{2} \bar{R}^{ab}_{\;\;\;\;n+1,n} Q_{ba}.
\end{align*} 

Therefore, (\ref{divergence}) becomes
\[ -|h_{n+1}|^2 Q_{n,n+1}  -\bar{R}^{ab}_{\;\;\;\;a,n+1}Q_{nb} + \frac{1}{2} \bar{R}^{ab}_{\;\;\;\;n+1,n}Q_{ba}. \]  Here we use the fact $\mbox{tr}_\sigma (h_{n+1})=0$. In the following, we apply the curvature formula (\ref{curvature tensor in terms of Q: Schwarzschild}) and an algebraic relation of components of $Q$ to simplify the last two terms. 

By \eqref{curvature tensor in terms of Q: Schwarzschild},
\begin{align*}
\bar{R}^{ab}_{\;\;\;\;a,n+1} &= -\frac{n(n-1)m}{r^{n+2}} Q^{ab}Q_{a,n+1} - \frac{n(n-2)m}{r^{n+2}} \lt( -Q^{ab}Q_{a,n+1} + Q^b_{\;\;n} Q_{n,n+1} \rt) \\
\bar{R}^{ab}_{\;\;\;\;n+1,n} &= -\frac{n(n-1)m}{r^{n+2}} \lt( \frac{2}{3} Q^{ab}Q_{n+1,n} - \frac{1}{3} Q^a_{\;\;n+1} Q_n^{\;\;b} - \frac{1}{3}Q^a_{\;\;n}Q^b_{\;\;n+1} \rt),
\end{align*}
and thus

\begin{align*}
&-\bar{R}^{ab}_{\;\;\;\;a,n+1}Q_{nb} + \frac{1}{2} \bar{R}^{ab}_{\;\;\;\;n+1,n}Q_{ba} \\
= & \frac{nm}{r^{n+2}} \Big[ (n-1) Q^{ab}Q_{a,n+1}Q_{nb} -(n-2)Q^{ab}Q_{a,n+1}Q_{nb} + (n-2)Q^{bn}Q_{nb}Q_{n,n+1} \\
&\quad - \frac{n-1}{3} (Q^{ab}Q_{ba}Q_{n+1,n} - Q^a_{\;\;n+1}Q_n^{\;\;b}Q_{ba}) \Big]. 
\end{align*}
From \cite[Lemma B.3]{thesis}, we have
\begin{align*}
Q^{ab}Q_{a,n+1}Q_{bn} = -\frac{1}{2}Q^{ab}Q_{ab}Q_{n,n+1}.
\end{align*}  
Therefore, 
\begin{align*}
\bar{R}^{ab}_{\;\;\;\;a,n+1}Q_{nb} - \frac{1}{2} \bar{R}^{ab}_{\;\;\;\;n+1,n}Q_{ba} = \frac{n(n-2)m}{r^{n+2}}  \lt( \frac{1}{2}Q^{ab}Q_{ab}Q_{n,n+1} + Q^b_{\;\;n}Q_{bn}Q_{n,n+1} \rt), 
\end{align*} 
and we obtain
\begin{align*}
0 &= \int_\Sigma \na_a \lt[ \lt( \mbox{tr}_\sigma(h_{n+1}) \sigma^{ab} - h_{n+1}^{\;\;\;\;\;\;ab} \rt) Q_{nb} \rt] d\mu \\
&= -\int_\Sigma \lt[ |h_{n+1}|^2 +\frac{n(n-2)m}{r^{n+2}} \lt(	\frac{1}{2}Q^{ab}Q_{ab} + Q^b_{\;\;n}Q_{bn} \rt) \rt] Q_{n,n+1} d\mu.
\end{align*}
If $m>0$, we obtain $Q_{ab} = Q_{bn}=0$. As $Q = r dr \wedge dt$, it is not hard to see that $\frac{\pl}{\pl t}$ is orthogonal to $\Sigma$. Hence $\Sigma$ lies on a time-slice. If $m=0$, we can only deduce $h_{n+1}=0$. However, this is the case of the Minkowski spacetime on which  Theorem \ref{Yau} is applicable. We conclude that $\Sigma$ lies in a totally geodesic hyperplane of the Minkowski spacetime. 
\end{proof}

\begin{remark}
Theorem \ref{global3} holds on a class of spherically symmetric spacetimes that satisfy null convergence condition. See \cite[Theorem 5.11]{thesis}.
\end{remark}
\section*{Acknowledgements}{P.-N. Chen is supported by NSF grant DMS-1308164 and M.-T. Wang is supported by NSF grants DMS-1105483 and
DMS-1405152. All authors would like to thank H. Bray for raising the question in the title and for inspiring discussions during his visit to Columbia Math Department on October 11, 2013.}
{
\end{document}
\begin{thebibliography}{99}  



\bibitem{Bray} H. Bray, \textit{Personal communications.}
\bibitem{Bray-Jauregui} H. Bray and J. Jauregui, \textit{Time flat surfaces and the monotonicity of the spacetime Hawking mass}, arXiv:1310.8638
\bibitem{Bray-Jauregui2} H. Bray and J. Jauregui, \textit{On curves with nonnegative torsion}, arXiv:1312.5171
\bibitem{Brown-York} J. D. Brown\ and\ J. W. York, Jr., Quasilocal energy in general relativity, in {\it Mathematical aspects of classical field theory (Seattle, WA, 1991)}, 129--142, Contemp. Math., 132, Amer. Math. Soc., Providence, RI.


\bibitem{Chen-Wang-Yau1} P. Chen, M.-T. Wang, and S.-T. Yau, \textit{Evaluating quasilocal energy and solving optimal embedding equation at null infinity}, Comm. Math. Phys. \textbf{308} (2011), no.3, 845--863.
\bibitem{Chen-Wang-Yau2} P. Chen, M.-T. Wang, and S.-T. Yau, \textit{Minimizing properties of critical points of quasi-local energy}, Comm. Math. Phys. {\bf 329} (2014), no. 3, 919--935.
\bibitem{do Carmo} M. P. do Carmo, {\it Differential geometry of curves and surfaces}, translated from the Portuguese, Prentice Hall, Englewood Cliffs, NJ, 1976.
\bibitem{Hawking} S. W. ~Hawking, {\it Gravitational radiation in an expanding universe}, J. Math. Phys., 9, 598--604, (1968).

\bibitem{JL}
J. Jezierski and M. {\L}ukasik.
\newblock J. Jezierski\ and\ M. \L ukasik, \textit{Conformal Yano-Killing tensor for the Kerr metric and conserved quantities}, Classical Quantum Gravity {\bf 23} (2006), no.~9, 2895--2918.
\bibitem{Liu-Yau}
C.-C. M. Liu\ and\ S.-T. Yau, \textit{Positivity of quasilocal mass}, Phys. Rev. Lett. {\bf 90} (2003), no.~23, 231102, 4 pp.
\bibitem{Miao-Tam-Xie} P. Miao, L.-F. Tam\ and\ N. Xie, \textit{Critical points of Wang-Yau quasi-local energy}, Ann. Henri Poincar\'e {\bf 12} (2011), no.~5, 987--1017.




    
\bibitem{Reilly} R. C. Reilly, \textit{Applications of the Hessian operator in a Riemannian manifold}, Indiana Univ. Math. J. {\bf 26} (1977), no.~3, 459--472.


\bibitem{Wang-Yau1} M.-T. Wang\ and\ S.-T. Yau, \textit{Quasilocal mass in general relativity}, Phys. Rev. Lett. {\bf 102} (2009), no.~2, no. 021101, 4 pp.
\bibitem{Wang-Yau2} M.-T. Wang\ and\ S.-T. Yau, \textit{Isometric embeddings into the Minkowski space and new quasi-local mass}, Comm. Math. Phys. {\bf 288} (2009), no.~3, 919--942.


\bibitem{thesis} Y.-K. Wang, {\it A Spacetime Alexandrov Theorem}, Thesis (Ph.D.)  Columbia University. 2014.
\bibitem{Weiner1} J. L. Weiner,  {\it Closed curves of constant torsion}, Arch. Math. (Basel) {\bf 25} (1974), 313--317.
\bibitem{Weiner2} J. L. Weiner, {\it Closed curves of constant torsion. II}, Proc. Amer. Math. Soc. {\bf 67} (1977), no. 2, 306--308. 
\bibitem{Yau} S. T. Yau, \textit{Submanifolds with constant mean curvature. I, II}, Amer. J. Math. {\bf 96} (1974), 346--366; ibid. {\bf 97} (1975), 76--100.

\end{thebibliography}
